\documentclass[11pt]{amsart}
\usepackage[letterpaper,vmargin=1in,hmargin=1.5in]{geometry}
\usepackage{amssymb}
\usepackage{enumitem}
\setlist{leftmargin=0.5cm}
\usepackage{tikz-cd}
\usepackage[pdfusetitle,colorlinks=true, linkcolor=blue, urlcolor=blue, citecolor=blue]{hyperref}

\numberwithin{equation}{section}

\theoremstyle{plain}
\newtheorem{thm}{Theorem}[section]
\newtheorem{lem}[thm]{Lemma}
\newtheorem{prop}[thm]{Proposition}
\newtheorem{conjecture}[thm]{Conjecture}
\theoremstyle{remark}
\newtheorem*{rem*}{Remark}
\newtheorem*{claim*}{Claim}

\newcommand{\CC}{\mathbb{C}}
\newcommand{\kk}{\mathbf{k}}
\newcommand{\iso}{\xrightarrow{\sim}}

\begin{document}

\title{Gaudin from Cherednik}
\author{Matthew McMillan}
\begin{abstract}
Let $W$ be a complex reflection group and $X$ the De~Concini--Procesi
wonderful model of its reflection arrangement. Inside Etingof's sheaf
of Cherednik algebras on $X$ we define a subsheaf generated by the
Dunkl operators that extend the standard ones on the hyperplane complement.
We conjecture that at $t=0$ the image of the logarithmic vector fields
in a natural module over this subsheaf is the sheaf of Gaudin algebras
on $X$ defined by Aguirre, extending a construction of Bonnaf\'e
and Rouquier from the hyperplane complement to $X$. We prove the conjecture
for $W=S_{3}$.
\end{abstract}
\maketitle

\section{Introduction}

Let $W\subset GL(V)$ be a complex reflection group with reflections
$S$, and let $V_{reg}\subset V$ be the complement of the reflection
hyperplanes $\ker\alpha_{s}$. For a $\CC W$-module $U$ and a parameter
$c$, the Knizhnik--Zamolodchikov (KZ) connection on $V_{reg}\times U$
is
\[
\nabla_{KZ}=d-c\sum_{s\in S}\varepsilon(s)\frac{d\alpha_{s}}{\alpha_{s}}\,s,\quad\text{with coefficients}\quad H_{\xi}(z)=c\sum_{s\in S}\varepsilon(s)\frac{\langle\xi,\alpha_{s}\rangle}{\alpha_{s}(z)}\,s
\]
for $\xi\in V$, where $\varepsilon=\det$. This connection is flat, and its monodromy
factors through the Hecke algebra of $W$. It was proposed in \cite{knizhnikCurrentAlgebraWessZumino1984}
for $W=S_{n}$, and extended to Coxeter groups in \cite{cherednikGeneralizedBraidGroups1989,cherednikComputationMonodromyCertain1990,cherednikMonodromyRepresentationsGeneralized1991}
and to complex reflection groups in \cite{broueComplexReflectionGroups1998}.
The $H_{\xi}$ generalize the Gaudin Hamiltonians: at each $z\in V_{reg}$
they span a space $G(z)\subset\CC S\subset\CC W$ of commuting elements,
the \emph{Gaudin algebra}.

The Dunkl operators $D_{\xi}=t\partial_{\xi}-H_{\xi}$, with $\mathcal{O}(V)$
and $W$, generate the rational Cherednik algebra $H_{t,c}(W,V)$ \cite{etingofSymplecticReflectionAlgebras2002}
whose modules at $t=1$ give bundles with KZ connection on $V_{reg}$
\cite{ginzburgCategoryMathcalORational2003,cherednikUnificationKnizhnikZamolodchikovDunkl1991}.
At $t=0$, Bonnaf\'e and Rouquier \cite{bonnafeCherednikAlgebrasCalogeroMoser2022}
recover the family $G(z)$ from $H_{0,c}$ and use its monodromy
to define Calogero--Moser cells, which conjecturally generalize Kazhdan--Lusztig
cells \cite{bonnafeCalogeroMoserKazhdan2013,bonnafeCherednikAlgebrasCalogeroMoser2022};
for $W=S_{n}$ this is proved in \cite{brochierGaudinAlgebrasRSK2023},
building on \cite{mukhinBetheAlgebraGaudin2014}.

Near the hyperplanes these structures are governed by the wonderful
model $p:X\to V$ of De~Concini and Procesi \cite{deconciniWonderfulModelsSubspace1995}:
$p$ is proper and birational, an isomorphism over $V_{reg}$, and
$D=X\setminus V_{reg}$ is a normal crossings divisor. In local coordinates
on $X$ the KZ system has a simple form and its monodromy is easily
computed \cite{deconciniHyperplaneArrangementsHolonomy1995}. This
geometry underlies Drinfeld's asymptotic zones \cite{drinfeldQuasitriangularQuasiHopfAlgebras1991}
and Varchenko's results relating KZ monodromy to the quantum $R$-matrix
and asymptotic KZ solutions to crystal bases \cite{varchenkoAsymptoticSolutionsKnizhnikZamolodchikov1995,varchenkoMultidimensionalHypergeometricFunctions1995}
(cf.~\cite{rybnikovCactusGroupMonodromy2018}). Aguirre \cite{aguirreNotionGaudinSubalgebras2014}
(cf.~\cite{aguirreGaudinSubalgebrasStable2011,aguirreGaudinSubalgebrasWonderful2016})
defined a morphism
$\mathcal{T}_{X}(-\log D)\to\mathcal{O}_{X}\otimes\CC S$ on logarithmic
vector fields with image a sheaf $\mathcal{G}au(W)$ of Gaudin algebras
on $X$ extending $G(z)$; a trigonometric analogue for root systems
is given in \cite{ilinBetheSubspacesWonderful2025}.

We propose a construction of $\mathcal{G}au(W)$ from Cherednik algebras,
generalizing that of \cite{bonnafeCherednikAlgebrasCalogeroMoser2022}
from the affine variety $V_{reg}$ to the non-affine model $X$. Etingof's
sheaf of Cherednik algebras $\mathcal{H}_{c}(X,W)$ on $X$ \cite{etingofCherednikHeckeAlgebras2017}
is generated by $\mathcal{O}_{X}$, $W$, and generalized Dunkl operators.
We define a subsheaf $\mathcal{H}_{c}^{\circ}(X,W)$ by keeping only
the sheaf $\mathcal{J}$ of those Dunkl operators that extend the
standard ones on $V_{reg}$, and set ${\mathcal{M}_{c}=\mathcal{H}_{c}^{\circ}/\mathcal{H}_{c}^{\circ}\cdot\mathcal{J}}$.
With $\kk=\CC[c]$, we conjecture that at $t=0$ the natural map ${\rho:\mathcal{O}_{X}\otimes\kk W}\to\mathcal{M}_{0,c}$
is an isomorphism (Conjecture~\ref{conj:rho}), and that the resulting
map $\mathcal{T}_{X}(-\log D)\to\mathcal{H}_{0,c}^{\circ}\to\mathcal{M}_{0,c}\cong\mathcal{O}_{X}\otimes\kk W$
is injective with image $\mathcal{G}au(W)$ (Conjecture~\ref{conj:main}).
\begin{thm}
\label{thm:intro}For $W=S_{3}$ acting on its reflection representation
$V=\CC^{2}$, with $X$ the blow-up of $V$ at the origin, Conjectures
\ref{conj:rho} and \ref{conj:main} hold.
\end{thm}

We expect the method to extend at least to $W=S_{n}$, and hope that
the $t\to0$ asymptotics of $\mathcal{M}_{t,c}$ (cf.~\cite{mukhinKZCharacteristicVariety2012})
will link the KZ system on $X$ to Calogero--Moser cells.

\section{Background}

\subsection{Cherednik algebra associated to $(W,V)$}

\subsubsection{Complex reflection group}

We work over $\CC$. Let $W$ be a complex reflection group acting
on the complex vector space $V$. We may consider $W\subset GL(V)$.
Let $S=\{s_{i}\}_{i}\subset W$ be the set of reflections. (By hypothesis
$W=\langle S\rangle$.) Assume $|S|=n$. We have a pairing $\langle-,-\rangle:V\times V^{*}\to\CC$.
For each $s\in S$ choose $\alpha_{s}\in V^{*}$ and $\alpha_{s}^{\vee}\in V$
such that $\alpha_{s}$ annihilates the hyperplane fixed by $s$,
i.e.~
\[
\ker\alpha_{s}=\ker(s-\mathrm{Id}_{V}),\qquad\CC\alpha_{s}^{\vee}=\mathrm{im}(s-\mathrm{Id}_{V}).
\]
We let $H_{s}=\ker\alpha_{s}$. This is the reflection hyperplane
associated to $s$.

Let $\varepsilon:W\to\CC^{\times}$ be the map of groups given by the
determinant. For $W$ a Coxeter group we have $\varepsilon(s)=-1$.

Let $V_{reg}=V\backslash\cup_{s\in S}\{H_{s}\}$. We define $X$ below.
We work with the \emph{$W$-equivariant Zariski topology} on $V$
and $X$. The open subsets are the Zariski open subsets that are also
$W$-equivariant. Then $\mathcal{O}_{V}\rtimes W$ and $\mathcal{D}_{V}\rtimes W$
are sheaves on $V$, and similarly $\mathcal{O}_{X}\rtimes W$ and
$\mathcal{D}_{X}\rtimes W$ are sheaves on $X$. (It may be possible
to rewrite many of the constructions in terms of sheaves on $X/W$,
but this approach, as used in \cite{bellamySmoothnessCentresRational2013}, is convenient for us.)

\subsubsection{Differential operators and sheaves}

Write $\mathcal{D}(V_{reg})$ for the algebra of differential operators
on $V_{reg}$ (and analogously for $V$). Then $\mathcal{D}(V_{reg})$
is a filtered algebra with filtration given by the order of a differential
operator. As a filtered algebra it has a graded Rees algebra $\mathrm{Rees}(\mathcal{D}(V_{reg}))$,
and we write $\hbar$ for the copy of $1$ in degree $1$.

The algebra $\mathcal{D}(V_{reg})$ on $V_{reg}$ considered as an
affine scheme determines by localization a sheaf of algebras $\mathcal{D}_{V_{reg}}$
on $V_{reg}$ such that $\mathcal{D}_{V_{reg}}(V_{reg})=\mathcal{D}(V_{reg})$.

An operator in $\mathcal{D}(V)$ can be restricted to an operator
in $\mathcal{D}(V_{reg})$. Since $V_{reg}$ is Zariski open in the
irreducible variety $V$, this restriction map is injective. There
is an open embedding $j':V_{reg}\to V$ and the restriction map gives
a injective morphism of sheaves $\mathcal{D}_{V}\to j'_{*}\mathcal{D}_{V_{reg}}$.

Later we will define $X$ and $X_{reg}$ and use the notation $j:X_{reg}\to X$
for that open embedding.

\subsubsection{Parameters}

We let $\kk=\CC[c]$ for $c$ an indeterminate.

\subsubsection{Standard Dunkl operators}

Let $R=\mathrm{Rees}\,\bigl(\kk\otimes(\mathcal{D}(V_{reg})\rtimes W)\bigr)$.
We have a map of complex vector spaces $V\to R$ given by
\[
V\ni\xi\mapsto D_{\xi}=\hbar\partial_{\xi}-c\sum_{s\in S}\varepsilon(s)\frac{\langle\xi,\alpha_{s}\rangle}{\alpha_{s}}s.
\]
The image of this map generates a left $\mathcal{O}(V_{reg})$-submodule
of $R$. By localizing this $\mathcal{O}(V_{reg})$-module at $W$-stable
affine open sets, we obtain a sheaf of left $\mathcal{O}_{V_{reg}}$-modules.
This will be called $\mathcal{J}_{\mid V_{reg}}$, the sheaf of `standard'
Dunkl operators on $V_{reg}$.

\subsubsection{Cherednik algebra}

The rational Cherednik algebra $H_{c}(W,V)$ associated to $(W,V)$
is the $\kk$-subalgebra of $R=\mathrm{Rees}\bigl(\kk\otimes(\mathcal{D}(V_{reg})\rtimes W)\bigr)$
generated by $x\in V^{*}$, $w\in W$, and $D_{\xi}\in\mathrm{Rees}\:\mathcal{D}(V_{reg})$
for $\xi\in V$.

Given a point $t\in\CC$, the (specialized) rational Cherednik algebra
$H_{t,c}(W,V)$ is the quotient $H_{c}(W,V)\big/(\hbar-t)$.

Recall that the quotient of a Rees algebra $A=\mathrm{Rees}(\sum F^{i}\bar{A})$
of a filtered algebra $\sum F^{i}\bar{A}$ by the ideal $\hbar A$
gives an algebra canonically isomorphic to the associated graded algebra
$\mathrm{gr}\,\bar{A}$. We know
\[
\mathrm{gr}\,\bigl(\mathcal{D}(V_{reg})\rtimes W\bigr)\iso\mathcal{O}(T^{*}V_{reg})\rtimes W,
\]
and it is easily checked that by specializing at $t=0$ we obtain
an embedding
\[
H_{0,c}(W,V)\hookrightarrow\kk\otimes\bigl(\mathcal{O}(T^{*}V_{reg})\rtimes W\bigr).
\]

\begin{rem*}
That this map is injective relies on the following general fact. Let
$A$ be a filtered algebra and $B\subset A$ be a filtered subalgebra.
Then we have a canonical embedding $\mathrm{Rees}(B)\hookrightarrow\mathrm{Rees}(A)$
and a canonical embedding $\mathrm{Rees}(B)\big/(\hbar)\hookrightarrow\mathrm{Rees}(A)\big/(\hbar)$.
The latter is, i.e., an embedding $\mathrm{gr}\,B\hookrightarrow\mathrm{gr}\,A$.
\end{rem*}
The image is generated by $x\in V^{*}$, $w\in W$, and
\[
D_{\xi}^{0}=\mathbf{L}_{\xi}-c\sum_{s\in S}\varepsilon(s)\frac{\langle\xi,\alpha_{s}\rangle}{\alpha_{s}}s
\]
for $\xi\in V$, where $\mathbf{L}_{\xi}=[\hbar\partial_{\xi}]$ is
the image of $\hbar\partial_{\xi}$ in the quotient $R/\hbar R$ that
is canonically identified with $\kk\otimes\bigl(\mathcal{O}(T^{*}V_{reg})\rtimes W\bigr)$.
\begin{rem*}
One can define $H_{t,c}$ by generators and relations, but a presentation
by generators and relations is not known for the sheaves of Cherednik
algebras that we consider.
\end{rem*}
There is a canonical isomorphism $\kk[V_{reg}\times V^{*}]\iso\kk\otimes\mathcal{O}(T^{*}V_{reg})$
given by $x\mapsto x$, $\xi\mapsto D_{\xi}^{0}$ for $x\in V^{*}$
and $\xi\in V$.

\subsection{Etingof-Cherednik sheaf of algebras on $X$}

\subsubsection{Model $X$}


There is an open embedding $j:X_{reg}\to X$. The restriction map
gives an isomorphism of sheaves $\mathcal{D}_{X}\iso j_{*}\mathcal{D}_{X_{reg}}$.

Recall that we have (extended) reflection hyperplanes $H_{s}'\subset X$
and $X_{reg}=X\backslash\cup_{s\in S}\{H_{s}'\}$. Each $H_{s}'$
is a proper transform of $H_{s}$.

We also have $V_{reg}\subsetneq X_{reg}$ and $X=\overline{X_{reg}}=\overline{V_{reg}}$.

\subsubsection{Etingof-Dunkl operators}

We define a sheaf $\mathcal{F}$ on $X$ of Etingof-Dunkl operators.
It is an $\mathcal{O}_{X}$-submodule generated inside the sheaf of
algebras $\mathcal{R}=\mathrm{Rees}\,\bigl(\kk\otimes(j_{*}\mathcal{D}_{X_{reg}}\rtimes W)\bigr)$
on $X$. Let $U\subset X$ be an affine open $W$-stable set. Let
$\xi\in\mathcal{T}_{X}(U)$ be a vector field on $U$. Then $\mathcal{F}(U)$
is generated as a left $\mathcal{O}_{X}(U)$-module by the elements
\[
D_{\xi}=\hbar\partial_{\xi}+c\sum_{(H_{s}',s)}f_{H_{s}'}(x)\cdot s.
\]
Here $f_{H_{s}'}(x)\in\Gamma(U,\mathcal{O}_{X}\langle H_{s}'\rangle)$
is an element of the coset $\chi_{H_{s}'}(\xi)\in\Gamma(U,\mathcal{O}_{X}\langle H_{s}'\rangle/\mathcal{O}_{X})$.
(We define $\chi_{H_{s}'}:\mathcal{T}_{X}\to\mathcal{O}_{X}\langle H_{s}'\rangle/\mathcal{O}_{X}$
elsewhere.)

Define $\mathcal{F}_{0}$ to be the image of $\mathcal{F}$ in the
quotient $\mathcal{R}/\hbar\mathcal{R}\cong\kk\otimes\bigl(\mathrm{gr}\,j_{*}\mathcal{D}_{X_{reg}}\rtimes W\bigr)$.

\subsubsection{Etingof's sheaf of algebras $\mathcal{H}_{c}(X,W)$
on $X$}

The sheaf of algebras $\mathcal{H}_{c}(X,W)$ on $X$ is by definition
the subsheaf of algebras of $\mathrm{Rees}\bigl(\kk\otimes(\mathcal{D}_{X_{reg}}\rtimes W)\bigr)$
generated by $\mathcal{O}_{X}$, $W$, and $\mathcal{F}$.

\subsection{Small Cherednik algebra on $X$}

We define a sheaf of algebras $\mathcal{H}_{c}^{\circ}(X,W)$
contained in $\mathcal{H}_{c}(X,W)$.
\begin{rem*}
The definition relies on the fact that $X$ is a model of $V_{reg}$,
and $V$ is a vector space equipped with an inner product.
\end{rem*}

\subsubsection{Good Dunkl operators}

We let `good' Dunkl operators be the
Etingof-Dunkl operators that extend standard Dunkl operators, as follows.
Standard Dunkl operators form a sheaf $\mathcal{J}_{\mid V_{reg}}$
of $\mathcal{O}_{V_{reg}}$-modules on $V_{reg}$. Let $j:V_{reg}\to X$
be the open embedding. Recall that $V_{reg}\subset X$ is Zariski
open. For any $W$-stable open affine $U\subseteq X$, we let
\[
\mathcal{J}(U)=\mathcal{F}(U)\cap j_{*}\mathcal{J}_{\mid V_{reg}}(U).
\]
This defines a sheaf $\mathcal{J}$ of operators on $X$ that we call
`good' Dunkl operators.

\subsubsection{Sheaf of small Cherednik algebras}

The sheaf of algebras $\mathcal{H}_{c}^{\circ}(X,W)$ on $X$ is by
definition the subsheaf of algebras of $\mathcal{R}=\mathrm{Rees}\bigl(\kk\otimes(j_{*}\mathcal{D}_{X_{reg}}\rtimes W)\bigr)$
generated by $\mathcal{O}_{X}$, $W$, and $\mathcal{J}$.

Let $\pi:T^{*}X_{reg}\to X_{reg}$ be the projection of varieties.
\begin{lem}
\label{lem:H^0-at-0}By specializing at $\hbar=0$, the inclusion
$\mathcal{H}_{c}^{\circ}(X,W)\hookrightarrow\mathcal{R}$ descends
to an embedding of algebras:
\[
\mathcal{H}_{0,c}^{\circ}(X,W)\hookrightarrow\kk\otimes\bigl(j_{*}\pi_{*}\mathcal{O}_{T^{*}X_{reg}}\rtimes W\bigr).
\]
\end{lem}

\section{Conjecture}

\subsection{Setup}

\subsubsection{Logarithmic vector fields}

Let $D=X\backslash X_{reg}$ be the divisor. Write $\mathcal{T}_{X}(-\log D)$
for the $\mathcal{O}_{X}$-submodule of $\mathcal{D}_{X}$ of logarithmic vector fields on
$X$ with the divisor $D$. It is the subsheaf of the tangent sheaf
$\mathcal{T}_{X}$ consisting of vector fields that are tangent to
a generic point of $D$. Recall that we have $j:\mathcal{D}_{X}\iso j_{*}\mathcal{D}_{X_{reg}}$
and may therefore view $\mathcal{T}_{X}(-\log D)$ in $j_{*}\mathcal{D}_{X_{reg}}$,
which we do now.

Observe that we have embeddings of $\mathcal{O}_{X}$-modules:
\[
\mathcal{T}_{X}(-\log D)\hookrightarrow j_{*}\mathcal{D}_{X_{reg}}\rtimes W\hookrightarrow\mathcal{R}=\mathrm{Rees}\bigl(\kk\otimes(j_{*}\mathcal{D}_{X_{reg}}\rtimes W)\bigr).
\]

Here a vector field $v$ over a $W$-stable affine open $U$, $v\in\mathcal{T}_{X}(-\log D)(U)$,
is mapped to $\hbar v$ considered as a differential operator in the
degree $1$ component of $\mathcal{R}(U)$.
\begin{prop}
The image of $\mathcal{T}_{X}(-\log D)$ in $\mathcal{R}$ lies in
$\mathcal{H}_{c}^{\circ}(X,W)$.
\end{prop}

\begin{proof}
This is clear in the case $W=S_{3}$, see below.
\end{proof}
Let $\psi:\mathcal{T}_{X}(-\log D)\to\mathcal{H}_{c}^{\circ}(X,W)$
denote the map determined by the proposition. Composing $\psi$ with
the projection to the quotient $\mathcal{R}/\hbar\mathcal{R}$ gives,
using Lemma \ref{lem:H^0-at-0}, an injection
\[
\mathcal{T}_{X}(-\log D)\overset{\psi_{0}}{\hookrightarrow}\mathcal{H}_{0,c}^{\circ}(X,W)\hookrightarrow\kk\otimes\bigl(j_{*}\pi_{*}\mathcal{O}_{T^{*}X_{reg}}\rtimes W\bigr).
\]

\subsubsection{Standard $\mathcal{H}_{c}^{\circ}$-module}

We define
\[
\mathcal{M}_{c}=\mathcal{H}_{c}^{\circ}(X,W)\Big/\mathcal{H}_{c}^{\circ}(X,W)\cdot\mathcal{J}
\]
and the specialization at $\hbar=t\in\CC$:
\[
\mathcal{M}_{t,c}=\mathcal{H}_{t,c}^{\circ}(X,W)\Big/\mathcal{H}_{t,c}^{\circ}(X,W)\cdot\mathcal{J}.
\]

\begin{rem*}
This module serves a role analogous to $\Delta^{reg}(\kk W)$ (but
for the model $X$ instead of $V_{reg}$) in the construction of Gaudin
Hamiltonians \cite[\S8.3.B]{bonnafeCherednikAlgebrasCalogeroMoser2022}. Let us review that role of $\Delta^{reg}(\kk W)$.
Recall the following composition of maps of $\kk[V_{reg}]$-modules:
\[
\kk[V_{reg}]\otimes\kk W\to H_{0,c}^{reg}\to H_{0,c}^{reg}\otimes_{\kk[V^{*}]\rtimes W}\kk W=\Delta^{reg}(\kk W).
\]
Here the first map is an inclusion, and the second is by $h\mapsto h\otimes1$.
The composition is an isomorphism $\varphi:\kk[V_{reg}]\otimes\kk W\iso\Delta^{reg}(\kk W)$.

The Gaudin operators arise from $\varphi$ in the following way. The
Dunkl operator embedding restricts to an isomorphism:
\[
H_{0,c}^{reg}\iso\kk[V_{reg}\times V^{*}]\rtimes W.
\]
We have $V\subset\kk[V_{reg}\times V^{*}]$ and the preimage of $V$
gives a subspace of $H_{0,c}^{reg}$. Taking the image of this in
$\Delta^{reg}(\kk W)$ and then the preimage under $\varphi$ yields
the Gaudin subalgebra.
\end{rem*}

\subsubsection{Parametrization of $\mathcal{M}_{0,c}$}

Let $i$ denote the embedding of $\mathcal{O}_{X}$-modules
\[
i:\mathcal{O}_{X}\otimes\kk W\to\mathcal{H}_{0,c}^{\circ}(X,W).
\]
Let $\rho=\pi\circ i$. Consider the following diagram:

\[
\begin{tikzcd}
  {\mathcal{T}_X(-\log D)} & {\mathcal{H}^\circ _{0,c}} & {\mathcal{O}_X\otimes \kk W} \\
  & {\mathcal{H}^\circ _{0,c} \Big/ \mathcal{H}^\circ _{0,c}\cdot \mathcal{J}}
  \arrow["{\psi_0}", from=1-1, to=1-2]
  \arrow["i"', from=1-3, to=1-2]
  \arrow["{\pi }"', from=1-2, to=2-2]
  \arrow["\rho", from=1-3, to=2-2]
\end{tikzcd}
\]

\begin{conjecture}
\label{conj:rho}The map $\rho$ is an isomorphism of $\mathcal{O}_{X}$-modules.
\end{conjecture}

\begin{proof}[Some ideas of a proof]
We could show $\rho$ is surjective using the fact (perhaps needs
Lemma) that an element of $\mathcal{H}_{0,c}^{\circ}(X,W)$ can be
written as a product of factors in $\mathcal{O}_{X}$, $W$, and $\mathcal{J}$
with factors in $\mathcal{J}$ on the right, which vanish in $\mathcal{M}_{0,c}$.
(The commutator of an element of $\mathcal{O}_{X}$ and an element
of $\mathcal{J}$ has no poles.)

An argument using the filtrations and the embedding
\[
\mathcal{H}_{0,c}^{\circ}(X,W)\hookrightarrow\kk\otimes\bigl(j_{*}\pi_{*}\mathcal{O}_{T^{*}X_{reg}}\rtimes W\bigr)
\]
should show injectivity.
\end{proof}

\subsubsection{Main conjecture}

The sheaf of Gaudin algebras is embedded as follows:
\[
\mathcal{G}au(W)\subset\mathcal{O}_{X}\otimes\mathfrak{t}_{W}^{1}\subset\mathcal{O}_{X}\otimes\kk W.
\]
Here $\mathfrak{t}_{W}^{1}$ is the complex vector space spanned
by the set of reflections $S\subset W$.
\begin{conjecture}
\label{conj:main}\leavevmode
\begin{itemize}
\item The map $\rho^{-1}\circ\pi\circ\psi_{0}$ is injective.
\item We have
\[
\mathrm{im}(\rho^{-1}\circ\pi\circ\psi)=\mathcal{G}au(W)\subset\mathcal{O}_{X}\otimes\mathfrak{t}_{W}^{1}.
\]
\end{itemize}
\end{conjecture}

\begin{proof}
One can see in the case $W=S_{3}$ that the composition $\rho^{-1}\circ\pi\circ\psi$
gives Gaudin Hamiltonians in $\mathcal{O}_{X}\otimes\kk W$.
\end{proof}

\section{The sheaf $\mathcal{G}au(S_{3})$ on $X$}

\subsection{$S_{3}$-invariant charts on $X$\label{subsec:Charts-on-the-blowup}}

In this section $W=S_{3}$. Set $\zeta=e^{2\pi i/3}$ throughout.
Let $V'=\oplus_{i=1}^{3}\CC e_{i}$ and $V'^{*}=\oplus_{i=1}^{3}\CC z_{i}$
with $z_{i}$ dual to $e_{i}$. Write $e_{ij}=e_{i}-e_{j}\in V'$
and $z_{ij}=z_{i}-z_{j}\in V'^{*}$. Let $V=\langle e_{12},e_{13}\rangle\subset V'$
and $V^{*}=V'^{*}/(z_{1}+z_{2}+z_{3})$. We also consider (the projection
of) $z_{ij}$ in $V^{*}$. The set $\{z_{12},z_{13}\}$ is a basis
of $V^{*}$. (It is not dual to $\{e_{12},e_{13}\}$.) We view $V=\mathrm{Spec}\;k[z_{12},z_{13}]$
as an affine scheme. We define a non-affine scheme $X$ by gluing
two open affine $S_{3}$-stable charts, $U_{-1}$ and $U_{\zeta}$.
They are given by:
\[
U_{-1}=\mathrm{Spec}\;k[\nu_{1},\nu_{2},(3\nu_{2}-1)^{-1}(3\nu_{2}-2)^{-1}]
\]
and
\[
U_{\zeta}=\mathrm{Spec}\;k[\mu_{1},\mu_{2},(1+\mu_{2}(\zeta^{2}-1))^{-1}].
\]
Define $\theta_{\mu}=1+\mu_{2}(\zeta-1)$ and $\theta_{\nu}=1+\nu_{2}(\zeta^{2}-1)$.

There is a distinguished open subvariety of $U_{-1}$ given by the
nonvanishing of $\theta_{\nu}(1+\nu_{2}(\zeta-1))$. Write $D_{-1}$
for that subvariety. Similarly there is a distinguished open subvariety
$D_{\zeta}$ of $U_{\zeta}$ given by the nonvanishing of $\theta_{\mu}((2\zeta+1)\mu_{2}-1)((\zeta+2)\mu_{2}-2)$.
There is an isomorphism of varieties $D_{-1}\iso D_{\zeta}$ determined
by the following Lemma:
\begin{lem}
\label{lem:gluing-alg-iso}There is an isomorphism of algebras
\begin{multline*}
k[\nu_{1},\nu_{2},\theta_{\nu}^{-1},(1+\nu_{2}(\zeta-1))^{-1},(3\nu_{2}-1)^{-1},(3\nu_{2}-2)^{-1}]\\
\iso k[\mu_{1},\mu_{2},\theta_{\mu}^{-1},(1+\mu_{2}(\zeta^{2}-1))^{-1},((2\zeta+1)\mu_{2}-1)^{-1},((\zeta+2)\mu_{2}-2)^{-1}]
\end{multline*}
given by
\[
\nu_{1}\mapsto\zeta^{2}\mu_{1}\theta_{\mu},\qquad\nu_{2}\mapsto\zeta\mu_{2}\theta_{\mu}^{-1},\qquad\theta_{\mu}\mapsto\theta_{\nu}^{-1}.
\]
The inverse is given by
\[
\mu_{1}\mapsto\zeta\nu_{1}\theta_{\nu},\qquad\mu_{2}\mapsto\zeta^{2}\nu_{2}\theta_{\nu}^{-1}.
\]
\end{lem}

\begin{proof}
Simple change of variables.
\end{proof}
On $U_{-1}\cap U_{\zeta}$ we let $\theta=\theta_{\mu}$ as an element
of $\mathcal{O}(U_{-1})$ and $\theta=\theta_{\nu}^{-1}$ as an element
of $\mathcal{O}(U_{\zeta})$.
\begin{rem*}
The lines $\nu_{2}=\infty$, $(3\nu_{2}-1)=0$, and $(3\nu_{2}-2)=0$
arise from the $2$-cycles, and the lines $\mu_{2}=\infty$ and $1+\mu_{2}(\zeta^{2}-1)=0$
arise from the $3$-cycles in $S_{3}$. The coordinates $\theta_{\mu}$
and $\theta_{\nu}$ give charts on $\mathbb{P}(V)$ in the standard
form.
\end{rem*}
The group $S_{3}$ acts on $X$, and the charts $U_{-1}$ and $U_{\zeta}$
are $S_{3}$-invariant. The action is given on $U_{-1}$ by:
\begin{align*}
(\nu_{1},\nu_{2})\overset{s_{12}}{\to} & \left(\nu_{1}(1-3\nu_{2}),\nu_{2}(3\nu_{2}-1)^{-1}\right)\\
(\nu_{1},\nu_{2})\overset{s_{13}}{\to} & \left(\nu_{1}(3\nu_{2}-2),(2\nu_{2}-1)(3\nu_{2}-2)^{-1}\right)\\
(\nu_{1},\nu_{2})\overset{s_{23}}{\to} & \left(\nu_{1},1-\nu_{2}\right),
\end{align*}
and on $U_{\zeta}$ by:
\begin{align*}
(\mu_{1},\mu_{2})\overset{s_{12}}{\to} & \left(\mu_{1}(1+\mu_{2}(\zeta^{2}-1)),-\mu_{2}(1+\mu_{2}(\zeta^{2}-1))^{-1}\right)\\
(\mu_{1},\mu_{2})\overset{s_{13}}{\to} & \left(\zeta\mu_{1}(1+\mu_{2}(\zeta^{2}-1)),-(\mu_{2}+\zeta)(1+\mu_{2}(\zeta^{2}-1))^{-1}\right)\\
(\mu_{1},\mu_{2})\overset{s_{23}}{\to} & \left(\zeta^{2}\mu_{1}(1+\mu_{2}(\zeta^{2}-1)),(1-\mu_{2})(1+\mu_{2}(\zeta^{2}-1))^{-1}\right).
\end{align*}
Define $\delta=z_{12}z_{13}z_{23}$ and $V_{reg}=\mathrm{Spec}\;k[z_{12},z_{13},\delta^{-1}]\subset V$.
This is an $S_{3}$-equivariant subvariety. We have $V_{reg}\hookrightarrow X$
using the formulas
\begin{align*}
z_{12} & =\nu_{1}\nu_{2},z_{13}=\nu_{1}(1-\nu_{2}),\text{ so }\nu_{1}=z_{12}+z_{13}\\
z_{12} & =\mu_{1}\mu_{2},z_{13}=\zeta^{2}\mu_{1}(1-\mu_{2}),\text{ so }\mu_{1}=z_{12}+\zeta z_{13}.
\end{align*}

\subsection{Definition of $\mathcal{G}au(S_{3})$}

\subsubsection{Definition of $\mathcal{G}au(S_{3})$ on $V_{reg}$ \label{subsec:Gau-Vreg}}

We define a sheaf of commutative Lie algebras $\mathcal{G}au(S_{3})_{\mid V_{reg}}$
over $V_{reg}$. Let $\mathcal{N}'$ be the free $\mathcal{O}(V_{reg})$-module
with basis the set of simple reflections $\{s_{12},s_{13},s_{23}\}=S\subset S_{3}$.
By localization, $\mathcal{N}'$ determines a sheaf over the affine
$V_{reg}$. The chosen basis determines three global sections, and
gives an isomorphism $\mathcal{O}_{V_{reg}}^{3}\iso\mathcal{N}'$.
Define three new global sections of $\mathcal{N}'$:
\begin{align*}
H_{1} & =-c\left(\frac{1}{z_{12}}s_{12}+\frac{1}{z_{13}}s_{13}\right)\\
H_{2} & =-c\left(-\frac{1}{z_{12}}s_{12}+\frac{1}{z_{23}}s_{23}\right)\\
H_{3} & =-c\left(-\frac{1}{z_{13}}s_{13}-\frac{1}{z_{23}}s_{23}\right)=-(H_{1}+H_{2}).
\end{align*}
These sections generate an $\mathcal{O}(V_{reg})$-submodule of $\mathcal{N}'(V_{reg})$.
By localization, this $\mathcal{O}(V_{reg})$-module determines a
sheaf $\mathcal{G}au(S_{3})_{\mid V_{reg}}$ over $V_{reg}$.
\begin{rem*}
Aguirre, Felder, and Veselov \cite{aguirreGaudinSubalgebrasWonderful2016},
following Vinberg, define the \emph{principal
Gaudin subalgebras} to be the $k$-subspaces $\langle H_{i}(z)\rangle_{i}\subset ks_{12}\oplus ks_{13}\oplus ks_{23}$
parametrized by points $z\in V_{reg}$, or the limits of such spans
for sequences of points in $V_{reg}$ approaching the hyperplanes.
We will consider their definition later on.
\end{rem*}

\subsubsection{Definition of $\mathcal{G}au(S_{3})$ on $X$}

Let $\mathcal{N}$ be the free $\mathcal{O}_{X}$-module with basis
$S$. So $\mathcal{N}_{\mid V_{reg}}=\mathcal{N}'$. We will define
$\mathcal{G}au(S_{3})(U_{-1})\subset\mathcal{N}(U_{-1})$ and $\mathcal{G}au(S_{3})(U_{\zeta})\subset\mathcal{N}(U_{\zeta})$
and glue these on $U_{-1}\cap U_{\zeta}$.

Define three sections of $\mathcal{N}(U_{-1})$:
\begin{align*}
G_{1} & =-c\Bigl((1-\nu_{2})s_{12}+\nu_{2}s_{13}\Bigr)\\
G_{2} & =-c\Bigl(-(1-2\nu_{2})s_{12}+\nu_{2}s_{23}\Bigr)\\
G_{3} & =-c\Bigl(-(1-2\nu_{2})s_{13}-(1-\nu_{2})s_{23}\Bigr).
\end{align*}
Let $\mathcal{G}au(S_{3})(U_{-1})\subset\mathcal{N}(U_{-1})$ be the
$\mathcal{O}(U_{-1})$-submodule generated by the $G_{i}$. Now define
three sections of $\mathcal{N}(U_{\zeta})$:
\begin{align*}
G'_{1} & =-c\Bigl((1-\mu_{2})s_{12}+\zeta\mu_{2}s_{13}\Bigr)\\
G'_{2} & =-c\Bigl(-(\mu_{2}+\zeta)s_{12}+\zeta^{2}\mu_{2}s_{23}\Bigr)\\
G'_{3} & =-c\Bigl(-(\mu_{2}+\zeta)s_{13}-\zeta(1-\mu_{2})s_{23}\Bigr).
\end{align*}
Let $\mathcal{G}au(S_{3})(U_{\zeta})\subset\mathcal{N}(U_{\zeta})$
be the $\mathcal{O}(U_{\zeta})$-submodule generated by the $G'_{i}$.
On the intersection $U_{-1}\cap U_{\zeta}$, using the morphism of
Lemma \ref{lem:gluing-alg-iso}, we compute $G_{i}=\zeta\theta^{-1}G'_{i}$.
Since $\theta^{-1}\in\mathcal{O}(U_{-1}\cap U_{\zeta})$ we see that in $\mathcal{N}(U_{-1}\cap U_{\zeta})$
the submodule generated by the $G_{i\mid U_{-1}\cap U_{\zeta}}$ agrees
with that generated by the $G'_{i\mid U_{-1}\cap U_{\zeta}}$. Therefore
the sheaves on the two affine opens glue to a sheaf on $X$ that is
a sub-$\mathcal{O}_{X}$-module of $\mathcal{N}$.
\begin{prop}
The sheaf $\mathcal{G}au(S_{3})$ of $\mathcal{O}_{X}$-modules (commutative
Lie algebras) just defined restricts to $\mathcal{G}au(S_{3})_{\mid V_{reg}}$
defined in \S\ref{subsec:Gau-Vreg}.
\end{prop}

\begin{proof}
Check $\mathcal{G}au(S_{3})(U_{-1})$ and $\mathcal{G}au(S_{3})(U_{\zeta})$
separately. In $\mathcal{G}au(S_{3})(U_{-1}\cap V_{reg})$ one computes
that
\begin{align*}
G_{1} & =\nu_{1}\nu_{2}(1-\nu_{2})H_{1}\\
G_{2} & =\nu_{1}\nu_{2}(1-2\nu_{2})H_{2}\\
G_{3} & =\nu_{1}(1-\nu_{2})(1-2\nu_{2})H_{3}.
\end{align*}
The coefficients of $H_{i}$ are units in $\mathcal{O}(U_{-1}\cap V_{reg})$,
so indeed the span of the $G_{i}$ is the span of the $H_{i}$ over
$\mathcal{O}(U_{-1}\cap V_{reg})$. Similarly, compute in $\mathcal{G}au(S_{3})(U_{\zeta}\cap V_{reg})$
that
\begin{align*}
G'_{1} & =\mu_{1}\mu_{2}(1-\mu_{2})H_{1}\\
G'_{2} & =\mu_{1}\mu_{2}(\mu_{2}+\zeta)H_{2}\\
G'_{3} & =\mu_{1}(1-\mu_{2})(\mu_{2}+\zeta)H_{3},
\end{align*}
and again the coefficients are units in $\mathcal{O}(U_{\zeta}\cap V_{reg})$,
so the span of the $G'_{i}$ is the span of the $H_{i}$ over $\mathcal{O}(U_{\zeta}\cap V_{reg})$.
\end{proof}

\section{Sheaf $\mathcal{G}au(S_{3})$ from $\mathcal{H}_{0,c}^{\circ}$ on
$X$}

\subsection{Etingof-Cherednik $\mathcal{H}_{c}(X,S_{3})$}

We record generators for the space of Etingof-Dunkl operators. The
$\mathcal{O}(U_{-1})$-module $\mathcal{F}(U_{-1})$ is generated
by:
\begin{align*}
D_{\nu_{1}} & =\hbar\partial_{\nu_{1}}\\
D_{\nu_{2}} & =\hbar\partial_{\nu_{2}}+c\left(\frac{s_{12}}{\nu_{2}}+\frac{s_{13}}{\nu_{2}-1}+\frac{s_{23}}{\nu_{2}-\frac{1}{2}}\right),
\end{align*}
and the $\mathcal{O}(U_{\zeta})$-module $\mathcal{F}(U_{\zeta})$
is generated by:
\begin{align*}
D_{\mu_{1}} & =\hbar\partial_{\mu_{1}}\\
D_{\mu_{2}} & =\hbar\partial_{\mu_{2}}+c\left(\frac{s_{12}}{\mu_{2}}+\frac{s_{13}}{\mu_{2}-1}+\frac{s_{23}}{\mu_{2}+\zeta}\right).
\end{align*}

Recall that $\Gamma(U_{-1},\mathcal{H}_{c})$ is generated by $\mathcal{F}(U_{-1})$,
$\mathcal{O}(U_{-1})$, and $W$, while $\Gamma(U_{\zeta},\mathcal{H}_{c})$
is generated by $\mathcal{F}(U_{\zeta})$, $\mathcal{O}(U_{\zeta})$,
and $W$.
\begin{rem*}
We have the following transition formulas over $U_{-1}\cap U_{\zeta}$:
\begin{align*}
D_{\nu_{1}} & =\zeta\theta^{-1}D_{\mu_{1}}\\
D_{\nu_{2}} & =\zeta^{2}\theta^{2}D_{\mu_{2}}+\theta(\zeta^{2}-1)\mu_{1}D_{\mu_{1}}.
\end{align*}
This shows that $D_{\nu_{2}}$ and $\nu_{2}D_{\nu_{2}}$ are globally
defined, but $D_{\nu_{1}}$ is not.
\end{rem*}

\subsection{Gaudin from Cherednik over $X$}

Write $\mathcal{H}_{0,c}$ for the sheaf $\mathcal{H}_{0,c}(X,S_{3})$.

\subsubsection{Good Dunkl operators}

The tangent vector fields $D_{e_{i}}$ defined on $V'$ give vector
fields on $V\subset V'$ by restriction. These may be written in terms
of the basis $\{\partial_{z_{12}},\partial_{z_{13}}\}$ of $TV$ arising
from the coordinate chart $\{z_{12},z_{13}\}$:
\[
D_{e_{1}}=\partial_{z_{12}}+\partial_{z_{13}},\;D_{e_{2}}=-\partial_{z_{12}},\;D_{e_{3}}=-\partial_{z_{13}}.
\]

The standard Dunkl operators are the following elements of $R=\mathrm{Rees}\,\bigl(\kk\otimes\mathcal{D}(V_{reg})\rtimes S_{3}\bigr)$:
\begin{align*}
\xi_{1} & =\hbar D_{e_{1}}+c\left(\frac{1}{z_{12}}s_{12}+\frac{1}{z_{13}}s_{13}\right)\\
\xi_{2} & =\hbar D_{e_{2}}+c\left(\frac{-1}{z_{12}}s_{12}+\frac{1}{z_{23}}s_{23}\right)\\
\xi_{3} & =\hbar D_{e_{3}}+c\left(\frac{-1}{z_{13}}s_{13}+\frac{-1}{z_{23}}s_{23}\right).
\end{align*}

These operators generate $\mathcal{J}(V_{reg})$ as a left $\mathcal{O}(V_{reg})$-module.
Define $y_{i}=[\hbar D_{e_{i}}]$ in the quotient $R/\hbar R$. The
same formulas but with $y_{i}$ substituted for $\hbar D_{e_{i}}$
give elements of $\kk\otimes\bigl(\mathcal{O}(T^{*}V_{reg})\rtimes S_{3}\bigr)$
which, together with $\mathcal{O}_{V}$ and $S_{3}$, generate $H_{0,c}(V,S_{3})$.

Now we consider the sheaf $\mathcal{J}$ of Etingof-Dunkl operators
that extend standard Dunkl operators. In the next two propositions
we understand the $\xi_{i}$ to be written in the variables $\nu_{i}$
or $\mu_{i}$ using the isomorphism of Lemma \ref{lem:gluing-alg-iso}.
\begin{lem}
The Dunkl operators generating $\mathcal{J}(U_{-1})$ over $\mathcal{O}(U_{-1})$
are:
\begin{align*}
D_{\nu_{1}} & =\hbar\nu_{1}\partial_{\nu_{1}}+c(s_{12}+s_{13}+s_{23})\\
D_{\nu_{2}} & =\hbar\partial_{\nu_{2}}+c\left(\frac{s_{12}}{\nu_{2}}+\frac{s_{13}}{\nu_{2}-1}+\frac{s_{23}}{\nu_{2}-\frac{1}{2}}\right).
\end{align*}
The Dunkl operators generating $\mathcal{J}(U_{\zeta})$ over $\mathcal{O}(U_{\zeta})$
are:
\begin{align*}
D_{\mu_{1}} & =\hbar\mu_{1}\partial_{\mu_{1}}+c(s_{12}+s_{13}+s_{23})\\
D_{\mu_{2}} & =\hbar\partial_{\mu_{2}}+c\left(\frac{s_{12}}{\mu_{2}}+\frac{s_{13}}{\mu_{2}-1}+\frac{s_{23}}{\mu_{2}+\zeta}\right).
\end{align*}
\end{lem}

\begin{proof}
Simple calculation. Notice the $\nu_{1}$ coefficient of $\partial_{\nu_{1}}$
and the $\mu_{1}$ coefficient of $\partial_{\mu_{1}}$.
\end{proof}
\begin{lem}
We have
\begin{align*}
\mathcal{J}(U_{-1}) & =\mathcal{O}(U_{-1})\cdot(\nu_{1}\xi_{1},\nu_{1}\xi_{2},\nu_{1}\xi_{3})\subset\mathcal{F}(U_{-1})\subset\mathcal{H}_{c}^{\circ}(U_{-1})
\end{align*}
and
\begin{align*}
\mathcal{J}(U_{\zeta}) & =\mathcal{O}(U_{\zeta})\cdot(\mu_{1}\xi_{1},\mu_{1}\xi_{2},\mu_{1}\xi_{3})\subset\mathcal{F}(U_{\zeta})\subset\mathcal{H}_{c}^{\circ}(U_{\zeta}).
\end{align*}
\end{lem}

\begin{proof}
Simple calculation.
\end{proof}
We will use the notation $p_{i}=[\hbar\partial_{\nu_{i}}]$ and $r_{i}=[\hbar\partial_{\mu_{i}}]$
for the `canonical momentum operators'
in the Rees quotient.

\subsubsection{Representation}

Recall the definitions
\[
\pi:\mathcal{H}_{0,c}^{\circ}\to\mathcal{M}_{0}=\mathcal{H}_{0,c}^{\circ}\big/\mathcal{H}_{0,c}^{\circ}\cdot\mathcal{J},
\]
and the inclusion
\[
i:\mathcal{O}_{X}\otimes\kk S_{3}\hookrightarrow\mathcal{H}_{0,c}^{\circ}(X,S_{3}).
\]
(Note that $i$ is injective because the degree $0$ part of $\mathcal{R}/\hbar\mathcal{R}$
is just $\mathcal{O}_{X}\rtimes S_{3}$.)
\begin{prop}
The morphism $\pi\circ i:\mathcal{O}_{X}\otimes\kk S_{3}\to\mathcal{M}_{0}$
is an isomorphism of $\mathcal{O}_{X}$-modules.
\end{prop}

\begin{proof}
First we check surjectivity. Consider $U_{-1}$. Every element of
$\mathcal{M}_{0}(U_{-1})$ has a representative in $\mathcal{H}_{c}^{\circ}(U_{-1})\subset\mathcal{R}(U_{-1})$
that is a sum of terms of the form
\[
\hbar^{r}f(\nu_{1},\nu_{2})sD_{\nu_{1}}^{\varepsilon_{1}}D_{\nu_{2}}^{\varepsilon_{2}}
\]
where $s\in S_{3}$ and $f\in\mathcal{O}(U_{-1})$ and $r,\varepsilon_{i}\in\mathbb{Z}^{\geq0}$.
(The commutators $[D_{\nu_{i}},D_{\nu_{j}}]$ and $[D_{\nu_{i}},\nu_{j}]$
have no poles. In fact the first type are zero in this case, in view
of the lemma.) The class of this representative is certainly zero
in $\mathcal{M}_{0}$ unless $r=\varepsilon_{1}=\varepsilon_{2}=0$.
All terms $f(\nu_{1},\nu_{2})s$ are contained in the image of $\pi\circ i$.

For injectivity, we view $\mathcal{J}_{0}$ inside $\kk\otimes\bigl(j_{*}\pi_{*}\mathcal{O}_{T^{*}X_{reg}}\rtimes S_{3}\bigr)$
using
\[
\mathcal{H}_{0,c}^{\circ}(X,W)\hookrightarrow\kk\otimes\bigl(j_{*}\pi_{*}\mathcal{O}_{T^{*}X_{reg}}\rtimes S_{3}\bigr)
\]
(again taking the quotient by $(\hbar)$) and consider the composite
morphism:
\[
\mathcal{O}_{X}\otimes\kk S_{3}\to\mathcal{M}_{0}\to\kk\otimes\bigl(j_{*}\pi_{*}\mathcal{O}_{T^{*}X_{reg}}\rtimes S_{3}\bigr)\Big/(\mathcal{J}_{0}).
\]

\begin{claim*}
The canonical map gives an isomorphism of $\mathcal{O}_{X}$-modules:
\[
\kk\otimes\bigl(j_{*}\mathcal{O}_{X_{reg}}\rtimes S_{3}\bigr)\iso\kk\otimes\bigl(j_{*}\pi_{*}\mathcal{O}_{T^{*}X_{reg}}\rtimes S_{3}\bigr)\Big/(\mathcal{J}_{0}).
\]
\end{claim*}
\begin{proof}
Using the standard coordinates on $T^{*}X$ we have on $U_{-1}$:
\begin{multline*}
\kk\otimes\bigl(j_{*}\pi_{*}\mathcal{O}_{T^{*}X_{reg}}\bigr)(U_{-1})\cong\kk[p_{1},p_{2}]\otimes\mathcal{O}(X_{reg}\cap U_{-1})\\
\cong\kk[p_{1},p_{2}]\otimes k[\nu_{1},\nu_{2},(3\nu_{2}-1)^{-1}(3\nu_{2}-2)^{-1},\nu_{2}^{-1},(\nu_{2}-1)^{-1},(2\nu_{2}-1)^{-1}].
\end{multline*}
It is clear that the two generators of $\mathcal{J}_{0}(U_{-1})$
are independent, and the quotient is $\kk\otimes\mathcal{O}(X_{reg}\cap U_{-1})\otimes S_{3}$
as a $\kk$-module. The result follows from this and similar reasoning
over $U_{\zeta}$.
\end{proof}
The composite morphism in question induces (via the inverse of the
isomorphism in the claim) the obvious map of modules $\mathcal{O}_{X}\otimes\kk S_{3}\to\kk\otimes\bigl(j_{*}\mathcal{O}_{X_{reg}}\rtimes S_{3}\bigr)$,
and that one is injective.
\end{proof}

\subsubsection{Logarithmic vector fields on $X$}

Let $D=X\backslash V_{reg}$ be the (normal crossing) divisor in the
closure of $V_{reg}$. Write $\delta_{\nu_{2}}=\nu_{2}(\nu_{2}-1)(\nu_{2}-\frac{1}{2})$
and $\delta_{\mu_{2}}=\mu_{2}(\mu_{2}-1)(\mu_{2}+\zeta)$. Note that
$D\cap U_{-1}$ is given by $\nu_{1}\delta_{\nu_{2}}$ and $D\cap U_{\zeta}$
is given by $\mu_{1}\delta_{\mu_{2}}$.

The logarithmic tangent vector fields on $U_{-1}$ and $U_{\zeta}$
are generated as follows:
\begin{align*}
\mathcal{T}_{X}(-\log D)(U_{-1}) & =\left\langle \nu_{1}\partial\nu_{1},\delta_{\nu_{2}}\partial\nu_{2}\right\rangle \\
\mathcal{T}_{X}(-\log D)(U_{\zeta}) & =\left\langle \mu_{1}\partial\mu_{1},\delta_{\mu_{2}}\partial\mu_{2}\right\rangle .
\end{align*}
When localized away from $\nu_{2}=1$ and $\nu_{2}=\frac{1}{2}$,
for example, the first span is equal to $\langle\nu_{1}\partial\nu_{1},\nu_{2}\partial\nu_{2}\rangle$,
as one expects.

\subsubsection{Gaudin from Cherednik}
\begin{prop}
The morphism $\rho^{-1}\circ\pi\circ\psi_{0}$ of $\mathcal{O}_{X}$-modules
is injective.
\end{prop}

\begin{proof}
Observe that $\nu_{1}p_{1}$ and $\delta_{\nu_{2}}p_{2}$ are independent
elements of $(\mathcal{R}/\hbar\mathcal{R})(U_{-1})$ that belong
to $\mathcal{H}_{0,c}^{\circ}(U_{-1})$. To see that $\nu_{1}p_{1}\in\mathcal{H}_{0,c}^{\circ}(U_{-1})$,
subtract $c(s_{12}+s_{13}+s_{23})$ from $D_{\nu_{1}}$. To see that
$\delta_{\nu_{2}}p_{2}\in\mathcal{H}_{0,c}^{\circ}(U_{-1})$, subtract
the appropriate expression (in $\kk\otimes\mathcal{O}(U_{-1})\rtimes S_{3}$)
from $\delta_{\nu_{2}}D_{\nu_{2}}$. These elements are also independent
when considered in $\mathcal{O}(U_{-1})\otimes\kk S_{3}$ through $\rho^{-1}\circ\pi$.
(It is equivalent to check that the generators of the Gaudin algebra
are independent.)

Apply a similar argument for $U_{\zeta}$.
\end{proof}
\begin{lem}
The image of $\rho^{-1}\circ\pi\circ\psi_{0}$ lies in $\mathcal{O}_{X}\otimes\mathfrak{t}_{S_{3}}^{1}\subset\mathcal{O}_{X}\otimes\kk S_{3}$.
\end{lem}

\begin{proof}
A trivial calculation that can be read off the formulas for (Dunkl)
generators of $\mathcal{J}_{0}(U_{-1})$ and $\mathcal{J}_{0}(U_{\zeta})$
and the basis of $\mathcal{T}_{X}(-\log D)(U_{-1})$ and $\mathcal{T}_{X}(-\log D)(U_{\zeta})$
given above.
\end{proof}
\begin{thm}
The image of $\rho^{-1}\circ\pi\circ\psi_{0}$ is $\mathcal{G}au(S_{3})$.
\end{thm}

\begin{proof}
The elements
\begin{align*}
A_{-1} & =c(s_{12}+s_{13}+s_{23})\\
B_{-1} & =\delta_{\nu_{2}}c\left(\frac{s_{12}}{\nu_{2}}+\frac{s_{13}}{\nu_{2}-1}+\frac{s_{23}}{\nu_{2}-\frac{1}{2}}\right)
\end{align*}
generate $\mathcal{G}au(S_{3})(U_{-1})$. Namely,
\begin{align*}
G_{1} & =+2\bigl(B_{-1}-\nu_{2}(\nu_{2}-1)A_{-1}\bigr)\\
G_{2} & =-2\bigl(B_{-1}-\nu_{2}(\nu_{2}-1/2)A_{-1}\bigr)\\
G_{3} & =+2\bigl(B_{-1}-(\nu_{2}-1)(\nu_{2}-1/2)A_{-1}\bigr).
\end{align*}
These elements clearly generate the image of $\rho^{-1}\circ\pi\circ\psi_{0}$
over $\mathcal{O}(U_{-1})$.

The elements
\begin{align*}
A_{\zeta} & =c(s_{12}+s_{13}+s_{23})\\
B_{\zeta} & =\delta_{\mu_{2}}c\left(\frac{s_{12}}{\mu_{2}}+\frac{s_{13}}{\mu_{2}-1}+\frac{s_{23}}{\mu_{2}+\zeta}\right)
\end{align*}
generate $\mathcal{G}au(S_{3})(U_{\zeta})$. Namely
\begin{align*}
G_{1}' & =B_{\zeta}-\mu_{2}(\mu_{2}-1)A_{\zeta}\\
G_{2}' & =-\bigl(B_{\zeta}-\mu_{2}(\mu_{2}+\zeta)A_{\zeta}\bigr)\\
G_{3}' & =B_{\zeta}-(\mu_{2}-1)(\mu_{2}+\zeta)A_{\zeta}.
\end{align*}
These elements clearly generate the image of $\rho^{-1}\circ\pi\circ\psi_{0}$
over $\mathcal{O}(U_{\zeta})$.
\end{proof}

\section*{Acknowledgements}

I thank Rapha\"el Rouquier for sharing his ideas, and for advice along
the way with many aspects of this project. This work was supported
by an NDSEG Research Fellowship from the US Department of Defense
and by the NSF through grant DMS-1702305. (An earlier version of this article was written in 2022 but not shared widely. This version has been minimally revised.) AI Transparency: AI was
used for references and to draft the introduction and abstract.

\end{document}